\documentclass{amsart}
\usepackage{amsmath,amssymb,amsfonts}
\usepackage{mathrsfs,latexsym,amsthm,enumerate}
\usepackage{tikz}
\usepackage{amscd}
\usepackage[all]{xy}

\newtheorem{lemma}{Lemma}[section]
\newtheorem{proposition}[lemma]{Proposition}
\newtheorem{corollary}[lemma]{Corollary}

\newtheorem{remark}[lemma]{Remark}

\newtheorem{theorem}[lemma]{Theorem}
\newtheorem{example}[lemma]{Example}

\begin{document}

\title[Invariant means]{Invariant means on Boolean inverse monoids}

\author{G. Kudryavtseva}
\address{Ganna Kudryavtseva,
Jo\v{z}ef \v{S}tefan Institute, Jamova cesta 39,
SI-1000, Ljubljana, SLOVENIA
and
Institute of Mathematics, Physics and Mechanics,
Jadranska ulica 19,
SI-1000,  Ljubljana, SLOVENIA}
\email{ganna.kudryavtseva@ijs.si}
\email{ganna.kudryavtseva@imfm.si}

\author{M. V. Lawson}
\address{Mark V. Lawson, 
Department of Mathematics
and the
Maxwell Institute for Mathematical Sciences, 
Heriot-Watt University,
Riccarton,
Edinburgh~EH14~4AS, UNITED KINGDOM}
\email{m.v.lawson@hw.ac.uk}

\author{D. H. Lenz}
\address{Daniel H. Lenz, Mathematisches Institut, 
Friedrich-Schiller Universit\"at Jena, 
Ernst-Abb\'{e} Platz~2, 
07743 Jena, GERMANY}
\email{daniel.lenz@uni-jena.de }

\author{P. Resende}
\address{Pedro Resende, 
Centre for Mathematical Analysis, Geometry, and Dynamical Systems,
Departamento de Matem\'atica,
Instituto Superior T\'ecnico,
Universidade de Lisboa, 
Av. Rovisco Pais,
1049-001 Lisboa, PORTUGAL}
\email{pmr@math.ist.utl.pt}

\thanks{This research was carried out in July 2014 under the auspices of the {\em Research in groups programme} of the International Centre for Mathematical Sciences (ICMS), Edinburgh.
We are grateful to the Scientific Director, Prof Keith Ball, and the Centre Manager, Ms Jane Walker, and all the staff at ICMS for their help and hospitality during our stay.
We would also like to thank Alistair Wallis for some fruitful discussions.
In addition, 
Kudryavtseva was also partially funded by the EU project TOPOSYS (FP7-ICT-318493-STREP) and by ARRS grant P1-0288;
Lawson was also partially supported by an EPSRC grant  (EP/I033203/1);
and 
Resende was also partially supported by FCT/Portugal through projects EXCL/MAT-GEO/0222/2012 and PEst-OE/EEI/LA0009/2013.
}

\begin{abstract} 
The classical theory of invariant means, 
which plays an important r\^ole in the theory of paradoxical decompositions, 
is based upon what are usually termed `pseudogroups'.
Such pseudogroups are in fact concrete examples of the Boolean inverse monoids which give rise to \'etale topological groupoids under non-commutative Stone duality.
We accordingly initiate the theory of invariant means on arbitrary Boolean inverse monoids.
Our main theorem is a characterization of when a Boolean inverse monoid admits an invariant mean.
This generalizes the classical Tarski alternative proved, for example, by de la Harpe and Skandalis, but using different methods.
\end{abstract}



\maketitle

\section{Introduction}

It is the thesis of this paper that inverse semigroups have  an important  r\^{o}le to play in measure theory in general, 
and the study of amenability in particular\footnote{There is a nice essay, available as arXiv:1306.2985, that also recognises the
importance of inverse semigroups in this context.  As an aside, the second author observes that the co-attribution of \cite{Lawson1998} given in this essay is incorrect.}.
In fact, the evidence for this is widespread but disguised.
If we survey the papers on amenability, we discover that they are based on the following concept:  a set of partial bijections closed under partial inverses, composition and restriction.
For example, \cite{Sherman} deals with what it calls sets of {\em piecewise translations} of a group
whereas \cite{HS} deals with what they call {\em pseudogroups} as does \cite{C} and \cite{Ch}.
The reason this concept is so important  is the fact that in the very definition of paradoxical decompositions, one is obliged to work with partial bijections.
Thus Ulam \cite{Ulam} states that in defining the notion of congruence between two subsets $A$ and $B$ of Euclidean space
one needs a one-to-one transformation from $A$ to $B$ but that this transformation {\em need only be defined on $A$ and not necessarily on the whole space}.
Accordingly, Tarski \cite{Tarski} himself devotes a chapter to partial automorphisms.
But the key point is that it has been known since the 1950's \cite[Chapter 1]{Lawson1998}, that such sets of partial bijections may be abstractly characterized as {\em inverse semigroups}:
indeed, there is a Cayley-type theorem, due to Wagner and Preston, that says that every abstract inverse semigroup is isomorphic to an inverse semigroup of partial bijections.
Thus, terminology to one side, inverse semigroups are already being used in the study of amenability.
But a trivial change in name is not the issue.
We believe that there is something deeper going on.
Recent work on non-commutative Stone dualities \cite{KL,LL,Resende} 
has linked  classes of inverse semigroups and classes of \'etale topological groupoids and, as the word duality suggests, this is a two-way relationship.
The r\^ole of \'etale topological groupoids within mathematics as a whole is well-established, particularly within the theory of $C^{\ast}$-algebras \cite{Renault, Paterson}.
Thus measure-theoretic type results for inverse semigroups will be connected to such results for \'etale groupoids.
But it is clear from the literature that, in formulating results, the inverse semigroup approach is the most natural.
In the remainder of this section, we shall describe the precise class of inverse semigroups we shall study and define the measure-theoretic notion whose theory we shall develop.
We refer the reader to \cite{Wagon} for the necessary background in paradoxical decompositions.

Recall that a {\em semigroup} is just a set equipped with an associative binary operation and a {\em monoid} is a semigroup with identity.
There is no special term for a semigroup with zero but all our inverse semigroups will be assumed to have a zero.
An invertible element in a monoid is called a {\em unit}.
The group of units of the monoid $S$ is denoted by $\mathsf{U}(S)$.
A semigroup $S$ is said to be {\em inverse} if for each $a \in S$ there exists a unique element $a^{-1}$ such that $a = aa^{-1}a$ and $a^{-1} = a^{-1}aa^{-1}$.
Standard properties of inverse semigroups are described in \cite{Lawson1998}
but we highlight here the key ones.
The set of idempotents of $S$, denoted by $E(S)$, forms a commutative idempotent subsemigroup.
The elements $a^{-1}a$ and $aa^{-1}$ are both idempotents.
A partial order $\leq$, called the {\em natural partial order}, is defined by $a  \leq b$ if, and only if, $a = be$ for some idempotent $e$.
With respect to this partial order $S$ is a partially ordered semigroup in which $a \leq b$ implies that $a^{-1} \leq b^{-1}$.
An inverse monoid in which each element is beneath a unit is said to be {\em factorizable}.
In an inverse semigroup with zero we have $0 \leq a$ for all elements $a$.
If there is no element properly between $0$ and $a$ we say that $a$ is an {\em atom}.
The set of idempotents forms a meet-semilattice with respect to this order and is therefore usually referred to as the {\em semilattice of idempotents}.
It is also an order ideal.
If $a,b \leq c$ then both $a^{-1}b$ and $ab^{-1}$ are idempotents.
Accordingly, if $a,b \in S$ define the {\em compatibility relation} $a \sim b$ if $a^{-1}b$ and $ab^{-1}$ are idempotents.
Thus being compatible is a necessary condition for two elements of an inverse semigroup to have a join.
If $S$ is an inverse semigroup with zero and both $a^{-1}b$ and $ab^{-1}$ are zero then we say that $a$ and $b$ are {\em orthogonal} and write $a \perp b$.
The basic examples of inverse semigroups are the {\em symmetric inverse monoids} $I(X)$ consisting of all partial bijections of the set $X$.
If $X$ is finite with $n$ elements we usually write just $I_{n}$.
Symmetric inverse monoids are factorizable precisely when they are finite.
The idempotents in $I(X)$ are the partial identities $1_{A}$ defined on subsets $A \subseteq X$;
the natural partial order is the restriction order of partial bijections;
the semilattice of idempotents is isomorphic to the Boolean algebra of all subsets of $X$.
We may now state precisely the Cayley-type theorem mentioned above by Wagner and Preston: 
every inverse semigroup is isomorphic to an inverse subsemigroup of some symmetric inverse monoid.
This result leads to a fruitful way of treating the elements of an abstract inverse semigroup.
Let $a \in S$.
We write $\mathbf{d}(a) = a^{-1}a$ and $\mathbf{r}(a) = aa^{-1}$ for the {\em domain} and {\em range} of $a$, respectively.
We also write
$$a^{-1}a \stackrel{a}{\longrightarrow} aa^{-1}.$$
An arrow defined in this way between two idempotents is equivalent to saying that the two idempotents are $\mathscr{D}$-related
in the usual sense of Green's relations.
Let $e$ and $f$ be idempotents.
Define $e \leq_{J} f$ if, and only if, $e \, \mathscr{D} \, e' \leq f$ for some idempotent $e'$.
If $e \leq_{J} f$ and $f \leq_{J} e$ then we have that $e \mathscr{J} f$ another of the familiar Green's relations.
Although $\mathscr{D} \subseteq \mathscr{J}$, we do not have equality in general.
In fact,  the equality $\mathscr{D} = \mathscr{J}$ can be viewed as an expression of the Schr\"oder-Bernstein theorem interpreted within inverse semigroup theory.
To see why, we work with partial bijections.
Then  $1_{A} \leq_{J} 1_{B}$ means that that there is an injection from $A$ to $B$.
Thus $1_{A} \, \mathscr{J} \, 1_{B}$ means that there are injections from $A$ to $B$ and from $B$ to $A$,
whereas $1_{A} \, \mathscr{D}\, 1_{B}$ means that there is a bijection from $A$ to $B$.

The distinction between monoids and semigroups is paralleled in what we term {\em unital} Boolean algebras as opposed to simply Boolean algebras, 
these usually being termed {\em generalized} Boolean algebras.
We say that an inverse semigroup with zero is {\em distributive} if it has all binary joins of compatible pairs of elements,
and multiplication distributes over such joins.
A distributive inverse semigroup is {\em Boolean} if its semilattice of idempotents is a Boolean algebra.
We are actually most interested in Boolean inverse {\em monoids} but our proofs will sometimes require Boolean inverse {\em semigroups}.
Henceforth, we shall usually just say {\em Boolean monoid} and {\em Boolean semigroup}.
{\em Morphisms} of Boolean monoids will be monoid homomorphisms that map zeros to zeros and preserve any binary joins that exist.
Such morphisms, restricted to the semilattices of idempotents, will therefore be morphisms of unital Boolean algebras.
The complement of an element $e$ in a unital Boolean algebra is denoted by $\bar{e}$.
The symmetric inverse monoids are examples of Boolean monoids.
If $S$ is a semigroup and $e$ is an idempotent then $eSe$ is a subsemigroup that is a monoid with respect to the identity $e$.
We call such subsemigroups {\em local monoids} of $S$.
If $S$ is a Boolean semigroup then for any idempotent $e \in S$ the local monoid $eSe$ is a Boolean monoid.
Classical Stone duality may be generalized to yield a duality between Boolean  monoids and a class of \'etale topological groupoids called {\em Boolean groupoids} \cite{LL,Resende,KL}.

The definitions above are all standard and well-known.
We now come to the key new definition of this paper which was suggested by \cite{C} and classical pseudogroup theory \cite{Plante}.
Let $S$ be a Boolean semigroup.
We shall say that $S$ has an {\em invariant mean} if there is a function $\mu \colon E(S) \rightarrow [0,\infty)$ satisfying the following conditions.
\begin{description}

\item[{\rm (IM1)}] For any $s \in S$, we have that $\mu (s^{-1}s) = \mu (ss^{-1})$.

\item[{\rm (IM2)}] If $e$ and $f$ are orthogonal idempotents we have that $\mu (e \vee f) = \mu (e) + \mu (f)$.

\end{description}
Of course, the constant function to zero is an invariant mean according to this definition.
We therefore explicitly exclude this possibility.
If $S$ is a Boolean semigroup and $e \in S$ is a distinguished idempotent, we say that an invariant mean $\mu$ is {\em normalized at $e$} if $\mu (e) = 1$.\\

\noindent
{\bf Convention. }If $S$ is a Boolean inverse {\em monoid} then we shall require our invariant means to be normalized at the identity unless stated otherwise.\\

The fundamental question that interests us in this paper is the existence or non-existence of invariant means on Boolean inverse {\em monoids}.
Our main theorem is Theorem~\ref{thm:TAS} where we describe necessary and sufficient conditions on a Boolean inverse monoid in order that it possess an invariant mean.
Our conditions generalize the classical {\em Tarski alternative} which is proved as a corollary to our main theorem in Theorem~\ref{thm:TA}.
We also describe some examples of Boolean inverse monoids that possess invariant means, 
and some natural conditions derived from the classical theory of paradoxical decompositions that ensure a Boolean inverse monoid have no invariant mean.

\section{Preliminary results and examples}

Our goal in this section is to demonstrate that the question of the existence or non-existence of an invariant mean on a Boolean monoid is interesting.

\subsection{Basics}

We begin with a few five-finger exercises.

\begin{lemma}\label{lem:thursday} Let $\mu$ be an invariant mean on the Boolean inverse semigroup $S$.
\begin{enumerate}

\item $\mu (0) = 0$.

\item  $\mu (e \vee f) = \mu (e) + \mu (f) - \mu (ef)$.

\item $e \leq f$ implies that $\mu (e) \leq \mu (f)$.

\item The set $I$ of all idempotents $e$ such that $\mu (e) = 0$ forms an ideal in $E(S)$
such that $s \in S$ and $e \in I$ imply that $ses^{-1} \in I$.

\item If $S$ is also a monoid and $\mu (1) = 1$ then $\mu(\bar{e}) = 1 - \mu(e)$.

\end{enumerate}
\end{lemma}
\begin{proof} (1) Since $0$ is orthogonal to itself $\mu (0) = 0$.

(2) Observe that $e \vee f = e \bar{f} \vee f \bar{e} \vee ef$, an orthogonal join.
Thus $\mu (e \vee f) = \mu (e \bar{f}) + \mu (f\bar{e}) + \mu (ef)$.
But 
$e \bar{f} \vee ef = e$ and so $\mu (e) = \mu (e\bar{f}) + \mu (ef)$,
and
$f \bar{e} \vee ef = f$ and so $\mu (f) = \mu (f \bar{e}) + \mu (ef)$.
The result now follows. 

(3) If $e \leq f$ then $f = e \vee f\bar{e}$ is an orthogonal join.
The result follows.

(4) Clearly, $0 \in I$.
It is immediate that $e,f \in I$  implies that $e \vee f \in I$.
Let $e \in I$ and $f \in E(S)$.
Then $ef \leq e$ and so $\mu (ef) \leq \mu (e) = 0$.
It follows that $ef \in I$.
To prove that $I$ is self-conjugate let $s \in S$ and $e \in I$.
Then $ses^{-1}$ is an idempotent.
Observe that $es^{-1}s \stackrel{se}{\longrightarrow} ses^{-1}$.
Thus $\mu (ses^{-1}) = \mu (es^{-1}s) \leq \mu (e) = 0$.
It follows that $\mu (ses^{-1}) = 0$.

(5) The join $1 = e \vee \bar{e}$ is orthogonal and so the result follows.
\end{proof}

The proof of the following is straightforward.

\begin{lemma}\label{lem:tuesday} Let $S$ be a Boolean semigroup with invariant mean $\nu$.
Let $e$ be any idempotent in $S$ such that $\nu (e) = r \neq 0$.
Define $\mu \colon eSe \rightarrow [0,1]$ by $\mu (a) = \frac{\nu (a)}{r}$.
Then $\mu$ is an invariant mean on the local monoid $eSe$ normalized at $e$.
\end{lemma}

Invariant means with the additional property that $\mu (e) = 0$ implies $e = 0$ are said to be {\em faithful}.

\begin{lemma}\label{lem:faithful} 
A Boolean inverse monoid equipped with a faithful invariant mean satisfies $\mathscr{D} = \mathscr{J}$.
\end{lemma}
\begin{proof} Let $\mu$ be the invariant mean.
Suppose that $e \, \mathscr{J} \, f$.
Then
$e \, \mathscr{D} \, i \leq f$ and $f \, \mathscr{D} \, j \leq e$.
Clearly $\mu (e) = \mu (i)$ and $\mu (f) = \mu (j)$.
But $\mu (i) \leq \mu (f)$ and $\mu (j) \leq \mu (e)$.
It follows that $\mu (e) = \mu (f)$.
We may write $f = i \vee f \bar{i}$, an orthogonal join, 
and $e = j \vee e \bar{j}$, an orthogonal join.
Thus $\mu (f) = \mu (i) + \mu (f \bar{i})$ and $\mu (e) = \mu (j) + \mu (e \bar{j})$.
By our calculations above, we have that 
$\mu (f \bar{i}) = 0$ and $\mu (e \bar{j}) = 0$.
We now use our assumption that the invariant mean is faithful to deduce $f \bar{i} = 0$ and $e \bar{j} = 0$.
It follows that $i = f$ and $j = e$ and so, in particular, $e \, \mathscr{D} \, f$.
\end{proof}

\subsection{Paradoxicality}

The problem of showing that a Boolean monoid does not have an invariant mean is intimately connected with the classical paradoxical decompositions.
The key definition is the following where we have adopted the terminology from \cite{HS}.
A Boolean inverse monoid is said to be {\em weakly paradoxical} if there exists a pair of elements $a$ and $b$ such that 
$\mathbf{d}(a) = 1 = \mathbf{d} (b)$ and $\mathbf{r}(a) \perp \mathbf{r}(b)$.

\begin{remark} 
{\em A Boolean inverse monoid is weakly paradoxical precisely when there is a  monoid embedding of the polycyclic monoid $P_{2}$ into $S$.
For more on the polycyclic inverse monoids see \cite{Lawson1998}.}
\end{remark}

The following is an abstract formulation of a classical result.
See \cite{HS}, for example.

\begin{lemma}\label{lem:sam} 
A weakly paradoxical Boolean monoid cannot have an invariant mean.
\end{lemma}
\begin{proof} Let  $a$ and $b$ be elements of $S$ such that 
$\mathbf{d}(a) = 1 = \mathbf{d} (b)$ and $\mathbf{r}(a) \perp \mathbf{r}(b)$.
Let $\mu$ be an invariant mean on $S$.
Then $\mu (\mathbf{r}(a)) = 1 = \mu (\mathbf{r}(b))$ and $\mu (\mathbf{r}(a) \vee \mathbf{r}(b)) = 2$.
But by Lemma~\ref{lem:thursday}      
$\mu (\mathbf{r}(a) \vee \mathbf{r}(b)) \leq 1$.
Thus the existence of an invariant mean in this case would lead to the conclusion that $2 \leq 1$.
\end{proof}

\begin{example}{\em The symmetric inverse monoid $I(\mathbb{N})$ has no invariant mean.
Denote by $\mathbb{E}$ and $\mathbb{O}$ the set of even and odd numbers, respectively.
The monoid  $I(\mathbb{N})$ contains the partial bijections $f \colon \mathbb{N} \rightarrow \mathbb{E}$ given by $n \mapsto 2n$
and $g \colon \mathbb{N} \rightarrow \mathbb{O}$ given by $n \mapsto 2n + 1$.
We now apply Lemma~\ref{lem:sam}.}
\end{example}

We shall now develop some ideas that will enable us to reformulate the definition of weakly paradoxical.
In a Boolean monoid $S$ an ideal $I$ is said to be a $\vee$-ideal if it is closed under binary compatible joins.
If the only $\vee$-ideals are the two trivial ones we say that $S$ is {\em $0$-simplifying}.
Let $e$ and $f$ be two non-zero idempotents in $S$.
Define $e \preceq f$ if and only if there exists a set of  elements $X = \{x_{1}, \ldots, x_{m}\}$ such that
$e = \bigvee_{i=1}^{m}  \mathbf{d}(x_{i})$
and 
$\mathbf{r}(x_{i}) \leq f$ for $1 \leq i \leq m$. 
We say that $X$ is a {\em pencil} from $e$ to $f$.
Clearly $e \preceq 1$ for every idempotent $e$.
The proof of the following is straightforward.

\begin{lemma}\label{lem:dingo} Suppose that  $e \preceq f$ in a Boolean monoid.
Then we may find a pencil $X$ from $e$ to $f$ where the domains of the elements of $X$ are pairwise orthogonal.
\end{lemma}

\begin{lemma}\label{lem:wren} 
If $I$ is a non-zero $\vee$-ideal where $e \in I$ and $f \preceq e$ then $f \in I$. 
\end{lemma}
\begin{proof} By definition, there is a pencil $\{x_{i}\}$ where $\mathbf{r}(x_{i}) \leq e$ and $f = \bigvee_{i=1}^{n} \mathbf{d}(x_{i})$.
But $ex_{i} = x_{i}$ and so $x_{i} \in I$, since $I$ is an ideal,
and similarly $\mathbf{d}(x_{i}) \in I$ for each $i$.
We now use the fact that $I$ is a $\vee$-ideal and so $f \in I$, as required.
\end{proof}

Define the equivalence relation $\equiv$ on the set of idempotents by putting $e \equiv f$ if, and only if, $e \preceq f$ and $f \preceq e$.
It can be proved that $S$ is $0$-simplifying if and only if $\equiv$ is the universal relation on the set of non-zero idempotents.
More on these matters, including proofs, can be found in \cite{Lawson2015}.
If $X \subseteq S$ is a non-empty subset, we denote by $X^{\vee}$ the set of all joins of finite non-empty compatible sets of elements of $X$.
clearly, $X \subseteq X^{\vee}$ and $(X^{\vee})^{\vee} = X^{\vee}$.
Let $e \in S$ be an idempotent.
Then $SeS$ is the principal ideal generated by $e$.
It can be checked that $(SeS)^{\vee}$ is a $\vee$-closed ideal and the smallest such ideal containing $e$.
We say that an idempotent $e$ is {\em large} if $S = (SeS)^{\vee}$.
This is an abstraction of the definition of {\em grande partie} given in \cite{HS}.
Result (2) below goes some way to justifying the use of the word `large'.

\begin{lemma}\label{lem:large} Let $S$ be a Boolean monoid.
\begin{enumerate}

\item An idempotent $e$ is large if, and only if, $1 \preceq e$.

\item If $e$ is a large idempotent then $\mu (e) > 0$ for every invariant mean $\mu$ on $S$.

\end{enumerate}
\end{lemma}
\begin{proof} (1) Suppose that $S = (SeS)^{\vee}$.
Then, in particular, $1 \in (SeS)^{\vee}$.
Thus there are idempotents $f_{1}, \ldots, f_{n} \in SeS$ such that $1 = f_{1} \vee \ldots \vee f_{n}$.
But each $f_{i} = a_{i}eb_{i}$.
Define $x_{i} = f_{i}a_{i}e$ and $y_{i} = eb_{i}f_{i}$.
Then $x_{i}$ and $y_{i}$ are mutually inverse.
Observe that $f_{i} = x_{i}y_{i}$ and $y_{i}x_{i} \leq e$.
We have therefore defined a pencil from $1$ to $e$, as required.
To prove the converse, we use Lemma~\ref{lem:wren} and observe that any $\vee$-ideal that contains $e$ must contain 1 and so is the whole of $S$.

(2) By definition, $1 \preceq e$ and so there is a pencil $\{a_{1}, \ldots, a_{n} \}$ such that $1 = \bigvee_{i=1}^{n} \mathbf{d}(a_{i})$, which we may assume is an orthogonal join, and $\mathbf{r}(a_{i}) \leq e$.
Now $\mu (\mathbf{r}(a_{i})) \leq \mu (e)$.
Thus $\sum_{i=1}^{n} \mu (\mathbf{r}(a_{i})) \leq n \mu (e)$.
But $\mu (\mathbf{r}(a_{i})) = \mu (\mathbf{d}(a_{i}))$.
It follows that $1 \leq n \mu (e)$ and so $\mu (e) \geq \frac{1}{n}$, as required.
\end{proof}

The following is the abstract reformulation of \cite[Proposition 2]{HS}.
It enables us to show that a Boolean monoid is weakly paradoxical by finding a single special element.

\begin{proposition}\label{prop:bike} Let $S$ be a Boolean monoid.
Then the following are equivalent.
\begin{enumerate}

\item $S$ is weakly paradoxical.

\item There exists an element $a \in S$ such that $\mathbf{d}(a) = 1$ and $\overline{\mathbf{r}(a)}$ is large.

\end{enumerate}
\end{proposition}
\begin{proof} (2)$\Longrightarrow$(1). Let $a$ be an element such that  $\mathbf{d}(a) = 1$ and $f = \overline{\mathbf{r}(a)}$ is large.
Thus $1 \preceq f$.
Therefore there exists a pencil $\{b_{1}, \ldots, b_{m} \}$ such that $1 = \bigvee_{i=1}^{n} \mathbf{d}(b_{i})$, 
an orthogonal join, 
and $\mathbf{r}(b_{i}) \leq f$.
Observe that 
$$\{f,afa^{-1}, \ldots, a^{m-1}f a^{-(m-1)}, a^{m}fa^{-m} \}$$ 
is an orthogonal set of idempotents.
To see why, put $f_{i} = a^{i-1}f a^{-(i-1)}$ and observe that $f_{i}f_{j}$ contains a factor $fa$ which is equal to zero because $f \perp \mathbf{r}(a)$.
Consider now the set of elements $\{b_{1},ab_{2}, \ldots, a^{m-1}b_{m}\}$.
Then the set of domains of these elements is the set $b_{i}^{-1}b_{i}$ and so is an orthogonal set whose join is 1.
Observe that $\mathbf{r}(a^{i-1}b_{i}) \leq f_{i}$ and so the set of ranges also forms an orthogonal set.
It follows that  $\{b_{1},ab_{2}, \ldots, a^{m-1}b_{m}\}$ is an orthogonal set.
Define $b = \bigvee_{i=1}^{m} a^{i-1}b_{i}$.
Then $\mathbf{d}(b) = 1$ and $\mathbf{r}(b) \perp \mathbf{r}(a^{m})$.
To see why the latter result holds, observe that $f_{i} \perp \mathbf{r}(a^{m})$
because the product $f_{i}\mathbf{r}(a^{m})$ contains $fa = 0$ as a factor.
It follows that $\mathbf{r}(b_{i}) \perp \mathbf{r}(a^{m})$.
This, by the definition of $b$, implies that $\mathbf{r}(b) \perp \mathbf{r}(a^{m})$.
Thus the set $\{a^{m},b\}$ proves that $S$ is weakly paradoxical.
The proof of the converse is {\em banale}.
\end{proof}

\begin{remark}{\em The above result may be viewed in the following light.
A set $X$ is {\em Dedekind infinite} if there is an injective map $f \colon X \rightarrow X$ whose image is a proper subset of $X$.
This is equivalent to saying that the Boolean inverse monoid $I(X)$ contains a copy of the bicyclic monoid $P_{1}$.
To say that a Boolean inverse monoid is weakly paradoxical can therefore be viewed as saying that it is infinite {\em but in a stronger sense}.
Thus $X$ is Dedekind infinite but the complement of the image of the injection $f$ that proves this is large.}
\end{remark}

The following is also immediate by Lemma~\ref{lem:sam} and Proposition~\ref{prop:bike},  but it is instructive to give a direct proof.

\begin{lemma}\label{lem:dahon} Let $S$ be a Boolean monoid.
If there exists an element $a \in S$ such that $\mathbf{d}(a) = 1$ and $\overline{\mathbf{r}(a)}$ is large then $S$ cannot have an invariant mean. 
\end{lemma}
\begin{proof}  Assume that there is an invariant mean $\mu$.
Then $\mu (\mathbf{r}(a)) = 1$ and so $\mu (\overline{\mathbf{r}(a)}) = 0$.
But by Lemma~\ref{lem:large}, this contradicts the assumption that $\overline{\mathbf{r}(a)}$ is large.
\end{proof}

In a $0$-simplifying Boolean monoid every non-zero idempotent is large.
We therefore have the following by Lemma~\ref{lem:faithful}.

\begin{corollary}\label{cor:tern} If a $0$-simplifying Boolean monoid has an invariant mean then all invariant means are faithful.
Thus a necessary condition that such a monoid have an invariant mean is that $\mathscr{D} = \mathscr{J}$.
\end{corollary}

We again adopt terminology from \cite{HS}.
A Boolean monoid is said to be {\em strongly paradoxical} if there exists a pair of elements $a$ and $b$ such that 
$\mathbf{d}(a) = 1 = \mathbf{d} (b)$ and $\mathbf{r}(a) \perp \mathbf{r}(b)$ and $1 = \mathbf{r}(a) \vee \mathbf{r}(b)$.

\begin{remark} 
{\em A Boolean monoid is strongly paradoxical precisely when there is a monoid embedding of the Cuntz inverse  monoid $C_{2}$ into $S$. 
This implies that the Thompson group $V$ is a subgroup of the group of units.
See \cite{Lawson2007}, \cite{LS}}.
\end{remark}

\begin{lemma}\label{lem:arden} A Boolean monoid in which  $\mathscr{D} = \mathscr{J}$ is weakly paradoxical if, and only if, it is strongly paradoxical.
\end{lemma}
\begin{proof} Only one direction needs proving.
Suppose that the Boolean inverse monoid is weakly paradoxical.
Thus there are elements $a$ and $b$ such that $\mathbf{d}(a) = 1 = \mathbf{d}(b)$
and $\mathbf{r}(a) \perp \mathbf{r}(b)$.
We have that $1 \, \mathscr{D} \, \mathbf{r}(b) \leq \overline{\mathbf{r}(a)}$.
Thus $1 \leq_{J}  \overline{\mathbf{r}(a)}$.
Since $\overline{\mathbf{r}(a)} \leq 1$ we also have that  $\overline{\mathbf{r}(a)} \leq_{J} 1$.
It follows that $1 \, \mathscr{J} \, \overline{\mathbf{r}(a)}$ and so,  by assumption, $1 \, \mathscr{D} \, \overline{\mathbf{r}(a)}$.
Thus there is an element $c$ such that $\mathbf{d}(c) = 1$ and $\mathbf{r}(c) = \overline{\mathbf{r}(a)}$.
Then $c \in S$ is such that $\mathbf{d}(c) = 1$ and $\mathbf{r}(a) \perp \mathbf{r}(c)$
and $\mathbf{r}(a) \vee \mathbf{r}(c) = 1$, showing that the inverse monoid is strongly paradoxical.
\end{proof}

The following lemma provides a situation where a Boolean monoid automatically satisfies $\mathscr{D} = \mathscr{J}$.
It is nothing other than the abstract version of a result from \cite{HS}.

\begin{lemma}\label{lem:mary} Let $S$ be a Boolean monoid whose Boolean algebra of idempotents has countable joins.
\begin{enumerate}

\item Let $1 \stackrel{a}{\longrightarrow} \mathbf{r}(a)$ where $\mathbf{r}(a) \neq 1$.
Let $e$ be any idempotent orthogonal to  $\mathbf{r}(a)$.
Then $1 \, \mathscr{D} \, \bar{e}$.

\item Let $i$ be any idempotent in $S$.
Let $i \stackrel{a}{\longrightarrow} \mathbf{r}(a)$ where $\mathbf{r}(a) < i$.
Let $e \leq i$ be any idempotent orthogonal to  $\mathbf{r}(a)$.
Then $i \, \mathscr{D} \, i\bar{e}$.

\item $\mathscr{D} = \mathscr{J}$.

\end{enumerate}
\end{lemma}
\begin{proof} (1)  Define $f = \bigvee_{k=0}^{\infty} a^{k}ea^{-k}$.
We prove that $f = e \vee afa^{-1}$ by showing that
$afa^{-1} = \bigvee_{k=0}^{\infty} a^{k+1}ea^{-(k+1)}$.
It is easy to see that
$\bigvee_{k=0}^{\infty} a^{k+1}ea^{-(k+1)} \leq afa^{-1}$.
Suppose that $a^{k+1}ea^{-(k+1)} \leq j$ for all $k$.
Then $a^{k}ea^{-k} \leq a^{-1}ja$ using the fact that $a^{-1}a = 1$.
It follows that $f \leq a^{-1}ja$ and so $afa^{-1} \leq j$, as required.
Observe also that $e \perp afa^{-1}$.
Thus $afa^{-1} = f\bar{e}$ and $e \leq f$.
Consider the element $af$.
Then $af \in fSf$.
Put $b = af \vee \bar{f}$, an orthogonal join.
Then $b^{-1}b = f \vee \bar{f} = 1$ and $bb^{-1} = f\bar{e} \vee \bar{f} = \bar{e}$,
where we use the fact that $\bar{f} \leq \bar{e}$.

(2) We simply apply the result in part (1) to the local monoid $iSi$.

(3) Suppose that $e \, \mathscr{D} \, f \leq e'$ and $e' \,  \mathscr{D} \, f' \leq e$.
Let $e \stackrel{a}{\longrightarrow} f$ and $e' \stackrel{b}{\longrightarrow} f'$.
The element $ba$ has the property that $\mathbf{d}(ba) = e$ and $\mathbf{r}(ba) \leq f' \leq e$.
We now apply part (2), with $\mathbf{r}(ba) \perp e\overline{f'}$.
It follows that there is an element $e \stackrel{c}{\longrightarrow} f'$.
Thus $e \stackrel{b^{-1}c}{\longrightarrow} e'$  and so $e \, \mathscr{D} \, e'$.
\end{proof}

\subsection{Examples of invariant means}

In this section, we shall construct some examples of invariant means on Boolean monoids.
As our starting point, we shall consider finite direct products of finite symmetric inverse monoids called {\em semisimple inverse monoids}.
Observe that since direct products of Boolean monoids are Boolean monoids it follows that semisimple inverse monoids are Boolean.
In what follows, $(x_1,\dots, x_p)^T$ etc denotes a column vector.

\begin{lemma}\label{lem:gear} \mbox{}
\begin{enumerate}

\item Finite symmetric inverse monoids have unique invariant means.

\item Let $S = I_{n(1)}\times I_{n(2)}\times \cdots\times I_{n(k)}$ be a semisimple inverse monoid.
Then invariant means on $S$ are in bijective  correspondence with non-negative vectors with real entries 
${\mathbf x}=(x_1,\dots, x_k)^T$ such that the following constraint equation 
$$
n(1)x_1+\cdots + n(k)x_k=1.
$$
holds.
Positive such vectors correspond to faithful invariant means.

\end{enumerate}
\end{lemma}
\begin{proof} We shall actually prove the general case (2), since (1) is then an immediate consequence.
The Boolean algebra $E(S)$ has $n(1)+ \cdots + n(k)$ atoms, 
where 
$n(1)$ atoms  correspond to the atoms of  $E(I_{n(1)})$,  
$n(2)$ atoms correspond to the atoms of  $E(I_{n(2)})$, and so on. 
Any two atoms corresponding to the same $E(I_{n(i)})$ are ${\mathscr{D}}$-related, 
so that $\mu$ has the same value on all the atoms corresponding to $E(I_{n(i)})$. 
If an atom $e$ belongs to the class corresponding to $E(I_{n(i)})$, 
we put $\mu(e)=x_i$. 
Since $\mu(1)=1$ and $1$ is a join of all the atoms, we obtain that the constraint equation
$$
n(1)x_1+\cdots + n(k)x_k=1.
$$
holds.
Conversely, assume that $x_1,\dots, x_k$ are non-negative reals that satisfy the constraint equation. 
Then we automatically have $x_i\leq 1$ for each $i$, since the coefficients $n(1)\dots, n(k)$ are positive integers. 
This data gives rise to an invariant mean for $S$ by putting $\mu(e)=x_i$, where $e$ belongs to the set of atoms corresponding to $E(I_{n(i)})$. 
This is well-defined since atoms corresponding to the different Boolean algebras  $E(I_{n(i)})$ are not ${\mathscr{D}}$-related.
\end{proof}

We shall generalize the examples in Lemma~\ref{lem:gear} to a much wider, and more interesting, class using some theory developed in \cite{LS}.
Recall that a morphism between two semisimple inverse monoids is a monoid homomorphism that maps zero to zero  and preserves all non-empty finite compatible joins.
Let 
$$S_{n_{1}} \stackrel{\tau_{1}}{\rightarrow} S_{n_{2}} \stackrel{\tau_{2}}{\rightarrow} S_{n_{3}} \stackrel{\tau_{3}}{\rightarrow} \ldots $$
be a sequence of semisimple inverse monoids and injective morphisms.
Then their direct limit $S = \varinjlim S_{n_{i}}$ is a factorizable Boolean inverse monoid called, by analogy with the case of $C^{\ast}$-algebras \cite{RLL}, an {\em AF inverse monoid} \cite{LS}.
If each of the semisimple inverse monoids in this direct limit is actually a finite symmetric inverse monoid then the direct limit 
is called a {\em UHF inverse monoid}\footnote{UHF inverse monoids can be classified using supernatural numbers just as in the $C^{\ast}$-algebra case \cite{RLL}.
Thus we can define $I_{n}$ where now $n$ is a supernatural number. A concrete representation of $I_{2^{\infty}}$ is constructed in \cite{LS}.}.

\begin{remark}{\em 
Although not needed here, observe that UHF inverse monoids are $0$-simplifying.
However, as with AF $C^{\ast}$-algebras, we do not believe that all $0$-simplifying AF inverse monoids are necessarily UHF.}
\end{remark}

Let $S$ and $T$ be Boolean monoids equipped with invariant means $\alpha$ and $\beta$, respectively.
Let $\theta \colon S \rightarrow T$ be a morphism.
We say that $\theta$ is {\em compatible} with these invariant means if $\beta (\theta (e)) = \alpha (e)$ for all idempotents $e$.
If $S$ and $T$ are semisimple inverse monoids then it is enough to check that the above equation holds for all idempotent atoms $e$.

\begin{lemma}\label{lem:anja1}  Let $S= I_{m(1)}\times \cdots\times I_{m(p)}$  and $T=I_{r(1)}\times \cdots \times I_{r(q)}$ be semisimple inverse monoids.
Define $\mathbf{m} = (m(1), \ldots, m(p))$ and $\mathbf{r} = (r(1), \ldots, r(q))$.
Let $S$ be equipped with the invariant mean $\alpha$ and let $T$ be equipped with the invariant mean $\beta$.
Let  $\tau \colon S \to T$ be an injective morphism.
Suppose that 
\begin{itemize}
\item The invariant mean $\alpha$ for $S$ is encoded by the vector ${\mathbf x}=(x_1,\dots, x_p)^T$;
\item The invariant mean $\beta$ for $T$ is encoded by the vector ${\mathbf y}=(y_1,\dots, y_q)^T$;
\item $M$ is the $q\times p$ matrix which determines $\tau$.  
Thus ${\mathbf{r}}=M{\mathbf{m}}$ holds (cf. \cite[Proposition 3.6]{LS}). 
\end{itemize}
Then $\tau$ is compatible with $\alpha$ and $\beta$ if, and only if, the equality ${\mathbf x}=M^T{\mathbf y}$ holds.
\end{lemma}
\begin{proof} Let $e$ be an atom of $E(S)$. 
Without loss of generality we may assume that it arises from an atom of $E(I_{n(1)})$. 
By the construction of the morphism determined by $M$, 
$\tau(e)$ is a join of $m_{11}$ atoms of $I_{r(1)}$, $m_{21}$ atoms of $I_{r(2)}$, and so on.
This means that $\beta\tau(e) = (m_{11},\dots, m_{q1})^T {\mathbf y}$.
If the equality $\beta\tau(e)=\alpha(e)$ holds for all atoms of $S$, it follows that we have the equalities
$$
x_k=(m_{1k},\dots, m_{qk})^T {\mathbf y}
$$
for all $1\leq k\leq p$. It follows that ${\mathbf x}=M^T{\mathbf y}$ holds.
The converse  is proved by reversing the arguments.
\end{proof}

The following lemma continues the above notation.

\begin{lemma}\label{lem:anja2} Let ${\mathbf x}=(x_1,\dots, x_p)^T$ and ${\mathbf y}=(y_1,\dots, y_q)^T$ be real vectors with non-negative entries such that the equality ${\mathbf x}=M^T{\mathbf y}$ holds. 
Then ${\mathbf{m}}^T{\mathbf x}=1$ if, and only if, ${\mathbf{r}}^T{\mathbf y}=1$. 
\end{lemma}
\begin{proof} ${\mathbf{m}}^T{\mathbf x}=1$ is, using the assumption, equivalent to ${\mathbf{m}}^TM^T{\mathbf y}=1$. 
Transposing this, we obtain the following equivalent equality ${\mathbf y}^TM{\mathbf{m}}=1$. 
Substituting $M{\mathbf{m}}$ with ${\mathbf{r}}$, the latter is rewritten as ${\mathbf y}^T{\mathbf{r}}=1$, which is equivalent to
${\mathbf{r}}^T{\mathbf y}=1$.
\end{proof}

We obtain the following result.

\begin{lemma}\label{lem:anja3} Let $S$ and $T$ be as above and let $\tau \colon S\to T$ be an injective morphism determined by the matrix $M$. 
Let $\beta$ be an invariant mean for $T$, and let ${\mathbf y}$ be its corresponding vector. 
Then the vector $M^T{\mathbf y}$ determines an invariant mean, $\alpha$, for $S$ and, moreover, $\tau$ is compatible with $\alpha$ and $\beta$.
\end{lemma}
\begin{proof}  Since $\beta$ is an invariant mean,  we have that  ${\mathbf{r}}^T{\mathbf y}=1$ by part (2) of Lemma~\ref{lem:gear}.
Lemma~\ref{lem:anja2} yields that ${\mathbf{m}}^T{\mathbf x}=1$ which, again by part (2) of Lemma~\ref{lem:gear},
means that the vector ${\mathbf x}$ encodes an invariant mean, $\alpha$, for $S$. 
By Lemma~\ref{lem:anja1}, the morphism $\tau$ is compatible with $\alpha$ and $\beta$.
\end{proof}

We now have the following theorem that generalizes our constructions on semisimple inverse monoids and finite symmetric inverse monoids. 

\begin{theorem}\label{thm:saturday} \mbox{}
\begin{enumerate}

\item Every AF inverse monoid can be equipped with an invariant mean, and in fact with a faithful invariant mean.

\item Every UHF inverse monoid is equipped with a unique invariant mean, necessarily faithful.

\end{enumerate}
\end{theorem}
\begin{proof} (1) Let $S$ be an AF inverse monoid and let 
$$S_0\stackrel{\tau_0}{\to} S_1\stackrel{\tau_1}{\to} S_2\stackrel{\tau_2}{\to}... $$ 
be a sequence of semisimple inverse monoids and injective morphisms defining $S$. 
We  have that $S=\bigcup_{i = 0}^{\infty} S_i$. 
We claim that to prove that $S$ is equipped with an invariant mean, 
it is enough to prove that there are invariant means $\mu_i$ defined on each $S_i$
such that the embedding $S_i\stackrel{\tau_i}{\to} S_{i+1}$ is compatible with $\mu_i$ and $\mu_{i+1}$. 
The reason being, that for each idempotent $e \in S$ we can then define $\mu(e)=\mu_i(e)$ where $e \in S_i$.
This is well-defined and is obviously an invariant mean for $S$. 
To verify our claim, we actually prove the following.
For every $n\geq 1$ and {\em for every} invariant mean $\mu_n$ for $S_n$ there are  invariant means $\mu_0,\dots, 
\mu_{n-1}$ for $S_0,\dots S_{n-1}$, respectively, such that $\tau_i$ is compatible with $\mu_i$ and $\mu_{i+1}$ for every $i=0, \dots, n-1$. 
We argue by induction on $n$. 
We first consider the base of the induction, that is, the case where $n=1$.  
We have $S_0=I_1$. 
The fact that $\tau_0$ is compatible with $\mu_0$ and $\mu_1$ is expressed by the equality $\mu_1\tau_0(1)=\mu_0(1)$ 
which automatically holds by the definition of the mean and the fact that $\tau_0$ is a monoid morphism. 
It follows that any invariant mean $\mu_1$ can be chosen for $S_1$.
The inductive step, that is a passage from $n=k$ to $n=k+1$, where $k\geq 1$, 
easily follows by applying Lemma~\ref{lem:anja3} to the embedding $S_k\stackrel{\tau_k}{\to} S_{k+1}$ and the inductive assumption.
The above argument can be easily adapted to prove that any AF inverse monoid can be equipped with a faithful invariant mean. 
This is because on semisimple inverse monoids faithful invariant means are encoded via positive vectors by part (2) of Lemma~\ref{lem:gear},
and the matrix $M$ encoding an injective morphism of semisimple inverse monoids has at least one positive entry in each column, 
so that whenever the vector ${\mathbf y}$ in the proof of the inductive step above is positive, the vector ${\mathbf x}=M^T{\mathbf y}$ is positive, too.

(2) It is enough to prove uniqueness.
Let 
$$I_{n_{1}} \stackrel{\tau_{1}}{\rightarrow} I_{n_{2}} \stackrel{\tau_{2}}{\rightarrow} I_{n_{3}} \stackrel{\tau_{3}}{\rightarrow} \ldots $$
be the sequence of finite symmetric inverse monoids and their morphisms that defines $S$ and regard $S = \bigcup_{i=1}^{\infty} I_{n_{i}}$ where $I_{n_{i}} \subseteq I_{n_{i+1}}$.
Let $\mu$ and $\nu$ be invariant means on $S$.
Let $e \in E(S)$.
Then $e \in I_{n_{i}}$ for some $i$.
Restricted to $I_{n_{i}}$ both $\mu$ and $\nu$ are invariant means and so must agree by the uniqueness of
invariant means on finite symmetric inverse monoids by part (1) of Lemma~\ref{lem:gear}.
Hence $\mu (e) = \nu (e)$ and so, since $e$ was arbitrary, $\mu = \nu$.
\end{proof}

\subsection{Amenability}

Finite groups are amenable and direct limits of finite groups are amenable.
Thus the groups of units of AF inverse monoids are amenable.
Whereas the AF inverse monoids are factorizable,
we shall focus on a class of Boolean monoids that satisfy a weaker condition than factorizability but where there is still a strong connection between
the structure of the group of units and the structure of the whole inverse monoid.
A Boolean monoid $S$ is said to be {\em piecewise factorizable} if each $s \in S$
may be written in the form $s = \bigvee_{i=1}^{n} g_{i}e_{i}$ where the $g_{i}$ are units and the $e_{i}$ are idempotents.
In this section, we shall prove that if a countable Boolean monoid has an amenable group of units then it has an invariant mean.
This generalizes the existence part of Theorem~\ref{thm:saturday}.

\begin{lemma}\label{lem:nice} Let $S$ be a piecewise factorizable Boolean monoid.
Then $S$ is equipped with an invariant mean if, and only if, there is a function $\sigma \colon E(S) \rightarrow [0,1]$ such that
\begin{enumerate}
\item $\sigma (1) = 1$.
\item $\sigma (geg^{-1}) = \sigma (e)$ for all $e \in E(S)$ and $g \in \mathsf{U}(S)$.
\item $\sigma (e \vee f) = \sigma (e) + \sigma (f)$ whenever $e$ and $f$ are orthogonal.
\end{enumerate}
\end{lemma}
\begin{proof} Suppose that $S$ is equipped with an invariant mean $\mu$.
Let $g$ be an arbitrary unit and $e$ and arbitrary idempotent.
Put $s = ge$.
Then by definition $\mu (s^{-1}s) = \mu (ss^{-1})$.
But $s^{-1}s = e$ and $ss^{-1} = geg^{-1}$.
It follows that conditions (1), (2) and (3) are satisfied.
Suppose now that we have a function $\sigma \colon E(S) \rightarrow [0,1]$ satisfying conditions (1), (2) and (3).
Let $s \in S$ be arbitrary.
Then $s = \bigvee_{i=1}^{n} s_{i}$ where $s_{i} = g_{i}e_{i}$  and the $g_{i}$ are units and the $e_{i}$ are idempotents.
Without loss of generality, we may assume that the union $s = \bigvee_{i=1}^{n} s_{i}$ is disjoint.
Thus $\mathbf{d}(s) =  \bigvee_{i=1}^{n} e_{i}$ is a disjoint union.
Likewise  $\mathbf{r}(s) = \bigvee_{i=1}^{n} \mathbf{r}(s_{i}) = \bigvee_{i=1}^{n} g_{i}e_{i}g_{i}^{-1}$ is a disjoint union.
We may apply properties (2) and (3) to deduce that $\sigma (s^{-1}s) = \sigma (ss^{-1})$.
It follows that $\sigma$ is an invariant mean.
\end{proof}

The above result is important because it tells us that as far as piecewise factorizable Boolean monoids $S$ are concerned
it is the action by conjugation of the group of units $\mathsf{U}(S)$ on the Boolean algebra $E(S)$ that is important.

\begin{proposition}\label{prop:amenability_one} Let $S$ be a countable Boolean monoid which is piecewise factorizable.
If the group of units of $S$ is amenable then $S$ is equipped with an invariant mean.
\end{proposition}
\begin{proof} It is a theorem of Bogolyubov, that a countable group is amenable if, and only if, for any continuous action 
of the group on a metrizable, compact space $X$ there exists a probability measure on $X$ which is invariant under the group action.
We refer the reader to \cite{GH} for references to this result and some background.
The action of $\mathsf{U}(S)$ on $E(S)$ by conjugation induces an action  by homeomorphisms on the Stone space $\mathsf{X}(E(S))$:
for each ultrafilter $F \subseteq E(S)$ define $g \cdot F = \{geg^{-1} \colon e \in F\}$.
It is easy to check that  $\mathsf{U}(S)$ acts by homeomorphisms on  $\mathsf{X}(E(S))$ 
and so the action $\mathsf{U}(S) \times \mathsf{X}(E(S)) \rightarrow  \mathsf{X}(E(S))$ is continuous.
The Stone space  $\mathsf{X}(E(S))$ is compact and it is countable so it is metrizable \cite[pp. 103--104]{K}.
It therefore follows by  Bogolyubov's theorem that there is a probability measure $\nu$ on $\mathsf{X}(E(S))$
which is invariant under the group action.
If $U_{e} = \{F \in \mathsf{X}(E(S)) \colon e \in F \}$ define $\mu (e) = \nu (U_{e})$.
We may now check that $\mu \colon E(S) \rightarrow [0,1]$ satisfies the conditions of Lemma~\ref{lem:nice}
and so there is an invariant mean on $S$.
\end{proof}

\begin{remark}{\em It is interesting to ask under what circumstances the converse to the above theorem is true.
For our purposes, it is enough to focus on the following situation.
Let $G$ be a countable discrete group acting faithfully by homeomorphisms on the Stone space $X$ of a countable Boolean algebra.
We want to know the circumstances under which the existence of an invariant mean {\em defined on the clopen subsets of $X$}
leads to the conclusion that $G$ is amenable.
In the case where $X$ has the discrete topology, then it is known \cite[Proposition~3.5]{Rosenblatt} that a sufficient condition
for $G$ to be amenable is that the stabilizer of each point of $X$ is amenable.
On the other hand, there are nice actions of free groups, famously non-amenable, which have invariant means \cite{vD}.
It is important to remember that when a group $G$ acts on a set $X$
the term {\em invariant mean} refers to maps defined on the power set of $X$.
However, see \cite[Theorem 10.8]{Wagon}, the invariant extension theorem.
The papers \cite{JM,JNS} suggest the delicate analyses that may be necessary to resolve this question.}
\end{remark}

\section{The existence of invariant means}

In this section, we shall determine necessary and sufficient conditions on a Boolean monoid in order that it have an invariant mean.
We shall do this by constructing from a Boolean monoid a commutative monoid called its {\em type monoid}.
We begin with some terminology.
Let $M$ be a commutative monoid with addition $+$ and identity $0$.
Define $a \leq b$ if, and only if, $b = a + c$ for some $c$.
We call this the {\em algebraic preorder} of the commutative monoid.
If $a \in M$ we abbreviate $\overbrace{a + \ldots + a}^{n}$ by $na$.
We say that $u$ is an {\em order-unit} if for each $a \in M$ there exists $n \geq 1$ such that $a \leq nu$.
In what follows, we shall regard the set $ [0,\infty)$ as a monoid under addition with $1$ as an order unit.
We shall be interested in monoid homomorphisms, which we shall call simply {\em morphisms}, from commutative monoids $M$ to the commutative monoid $[0,\infty)$.
If $u \in M$ is an order unit, then we shall be interested in morphisms that map $u$ to $1$.
That is,  distinguished order units are mapped one to the other.
We shall say that such morphisms are {\em normalized at $u$}.
The key theorem we shall prove in this section is the following.

\begin{theorem}[The type monoid]\label{thm:type-monoid} Let $S$ be a Boolean monoid.
\begin{enumerate}

\item There is a commutative monoid $\mathsf{T}(S)$, called the {\em type monoid} of $S$,
equipped with a map $\delta \colon E(S) \rightarrow \mathsf{T}(S)$ such that the following properties hold.
\begin{enumerate}
\item $\delta (0)$ is the identity.
\item  $\delta (1)$ is an order-unit.
\item  If $e,f \in E(S)$ are othogonal then $\delta (e \vee f) = \delta (e) + \delta (f)$. 
\item If $e \, \mathscr{D} \, f$ then $\delta (e) = \delta (f)$.
\end{enumerate}

\item There is a bijective correspondence between the set of morphisms from  $\mathsf{T}(S)$ to $[0,\infty)$ normalized at $\delta (1)$ 
and the set of invariant means on the Boolean monoid $S$.
\end{enumerate}
\end{theorem}

The rationale for the above theorem is the following proved as \cite[Theorem 9.1]{Wagon}.

\begin{theorem}[Tarski's theorem]\label{them:wagon} Let $T$ be a commutative monoid with distinguished order-unit $u$.
Then $T$ admits a morphism to $[0,\infty)$ normalized at $u$ if, and only if, for all natural numbers $n \geq 0$,
we have that  $$(n+1)u \nleq nu.$$ 
\end{theorem}

By using the internal description of the type monoid of a Boolean monoid $S$, 
we may use Tarski's theorem to deduce algebraic necessary and sufficient conditions on $S$ 
in order that it possess an invariant mean.

Although we are primarily interested in the case of Boolean inverse {\em monoids} there are, as we shall see, good reasons to describe
the construction in the slightly more general setting of Boolean inverse {\em semigroups}.
Let $S$ be a Boolean semigroup.
Put $\mathsf{E}(S) = E(S)/\mathscr{D}$.
Denote the $\mathscr{D}$-class containing the idempotent $e$ by $[e]$.
Define 
$[e] \oplus [f]$ as follows.
Suppose that we can find idempotents $e' \in [e]$ and $f' \in [f]$ such that $e'$ and $f'$ are orthogonal.
Then define $[e] \oplus [f] = [e' \vee f']$.
Otherwise, the operation $\oplus$ is undefined.
We write $\exists [e] \oplus [f]$ to mean that $[e] \oplus [f]$ is defined.
It is convenient to put $\mathbf{0} = [0]$ and, if $S$ is actually a monoid, to put $\mathbf{1} = [1]$.
The proofs of the following were given in \cite{LS}, but we repeat them here for the sake of completeness.
We shall also need the following definition.
A Boolean semigroup $S$ is said to be {\em orthogonally separating} if for all $e,f \in E(S)$
there exist orthogonal idempotents $e'$ and  $f'$ such that $e \, \mathscr{D} \, e'$ and $f \, \mathscr{D} \, f'$.
The following is based on a construction first sketched in \cite{Renault} and developed  and proved in  \cite{LS}.
We nevertheless give full proofs for the sake of completeness.

\begin{proposition}\label{prop:cakes} Let $S$ be a Boolean semigroup.
\begin{enumerate}

\item The partial operation $\oplus$ is well-defined.

\item $\exists [e] \oplus [f]$ if, and only if, $\exists [f] \oplus [e]$, and they are equal.

\item $\exists ([e] \oplus [f]) \oplus [g]$ if, and only if, $\exists [e] \oplus ([f] \oplus [g])$, and they are equal.

\item $[0] \oplus [e]$ always exists and equals $[e]$.

\item $[e] \leq [f]$ if, and only if, $e \, \mathscr{D}\, i \leq f$ for some idempotent $i$.
That is, if, and only if, $e \leq_{J} f$.

\item If $S$ is a monoid, then for each element $[e] \in \mathsf{E}(S)$ we have that $[e] \leq \mathbf{1}$.
It follows that $\mathbf{1}$ is a top.

\item The algebraic preorder on $\mathsf{E}(S)$ is an order if, and only if, $\mathscr{D} = \mathscr{J}$ in $S$.

\item The operation $\oplus$ is everywhere defined if, and only if, $S$ is orthogonally separating.

\end{enumerate}
\end{proposition}
\begin{proof} (1) Let 
$e' \, \mathscr{D} \, e''$
and 
$f' \, \mathscr{D} \, f''$
where $e'$ is orthogonal to $f'$, and $e''$ is orthogonal to $f''$.
We prove that $e' \vee f' \, \mathscr{D} \, e'' \vee f''$.
By assumption, there are elements $e' \stackrel{a}{\longrightarrow} e''$ and $f' \stackrel{b}{\longrightarrow} f''$.
The elements $a$ and $b$ are orthogonal and so $a \vee b$ exists.
But $e' \vee f' \stackrel{a \vee b}{\longrightarrow} e'' \vee f''$. 

(2) Immediate.

(3) Suppose that 
$\exists ([e] \oplus [f]) \oplus [g].$ 
Then 
$\exists [e] \oplus [f]$
and so we may find
$e \stackrel{a}{\longrightarrow} e'$ and $f \stackrel{b}{\longrightarrow} f'$
such that $e'$ and $f'$ are orthogonal.
By definition, $[e] \oplus [f] = [e' \vee f']$.
Since $\exists [e' \vee f'] \oplus [g]$,
we may find $e' \vee f' \stackrel{c}{\longrightarrow} i$
and
$g \stackrel{d}{\longrightarrow} g'$
such that $i$ and $g'$ are orthogonal.
It follows that 
$$([e] \oplus [f]) \oplus [g] = [i \vee g'].$$
Define $x = ce'$ and $y = cf'$.
Then
$$e' \stackrel{x}{\longrightarrow} \mathbf{r}(x)
\mbox{ and }
f' \stackrel{y}{\longrightarrow} \mathbf{r}(y).$$
Since $i$ is orthogonal to $g'$ and $\mathbf{r}(y) \leq i$,
we have that $\mathbf{r}(y)$ and $g'$ are orthogonal.
In addition, $yb$ has domain $f$ and range $\mathbf{r}(y)$.
It follows that $\exists [f] \oplus [g]$ and it is equal to $[\mathbf{r}(y) \vee g']$.
Observe next that $\mathbf{r}(x)$ is orthogonal to $\mathbf{r}(y)$ and, since $\mathbf{r}(x) \leq i$ it is also orthogonal to $g'$.
It follows that $\mathbf{r}(x)$ is orthogonal to $\mathbf{r}(y) \vee g'$.
But $xa$ has domain $e$ and range $\mathbf{r}(x)$.
It follows that
$\exists [e] \oplus [\mathbf{r}(y) \vee g']$ is defined
and equals 
$[\mathbf{r}(x) \vee \mathbf{r}(y) \vee g']$.
But $\mathbf{r}(x) \vee \mathbf{r}(y) = i$.
It follows that we have shown
$$\exists [e] \oplus ([f] \oplus [g])$$
and that it equals
$([e] \oplus [f]) \oplus [g]$.
The reverse implication follows by symmetry.


(4) Immediate.

(5) Suppose that 
$e \stackrel{x}{\longrightarrow} i \leq f$.
We may find an idempotent $j$ such that $f = i \vee j$ and $i \wedge j = 0$.
Then $[e] \oplus [j] = [f]$ and so $[e] \leq [f]$.
Conversely, 
suppose that $[e] \leq [f]$ where $e$ and $f$ are idempotents.
Then there exists an idempotent $g$ such that $[e] \oplus [g] = [f]$.
By definition, there are elements $e \stackrel{a}{\longrightarrow} e'$ and $g \stackrel{b}{\longrightarrow} g'$
such that $e' \vee f' \, \mathscr{D} \, f$.
But then $e \, \mathscr{D} \, e' \leq f$, as required.


(6) Denote by $\bar{e}$ the complement of $e$ in the Boolean algebra of idempotents of $S$.
Then $e \vee \bar{e} = 1$ is an orthogonal join.
It follows that $\exists [e] \oplus [\bar{e}]$ and $[e] \oplus [\bar{e}] = \mathbf{1}$.
Thus $[e] \leq \mathbf{1}$.

(7) Suppose that $\mathscr{D} = \mathscr{J}$.
If $[e] \leq [f]$ and $[f] \leq [e]$ then $e \leq_{J} f$ and $f \leq_{J} e$ and so $e\,  \mathscr{J} \, f$.
By assumption, $e \, \mathscr{D} \, f$ and so $[e] = [f]$.
Conversely, suppose that $\leq$ is an order.
Let $e\,  \mathscr{J} \, f$.
Then $[e] \leq [f]$ and $[f] \leq [e]$ and so, by assumption, $[e] = [f]$.
Thus $e \, \mathscr{D} \, f$, as required.

(8) This is immediate from the definition.
\end{proof}


We highlight the fact that when $S$ is orthogonally separating,  $(\mathsf{E}(S),\oplus, \mathbf{0})$ is a commutative monoid, in which case we write $+$ rather than $\oplus$.

Let $S$ be an arbitrary Boolean monoid.
There is no reason for $S$ to be orthogonally separating, but we shall prove that $S$ may be embedded into a Boolean semigroup that is.
Let $S$ be a Boolean monoid and let $m$ and $n$ either be finite non-zero natural numbers or both equal to the first infinite ordinal $\omega$. 
An {\em $m \times n$ generalized rook matrix over $S$} satisfies the following three conditions:
\begin{description}

\item[{\rm (RM1)}] If $a$ and $b$ are in distinct columns and lie in the same row of $A$ then $a^{-1}b = 0$. That is $\mathbf{r}(a) \perp \mathbf{r}(b)$.

\item[{\rm (RM2)}] If $a$ and $b$ are in distinct rows and lie in the same column of $A$ then $ab^{-1} = 0$. That is $\mathbf{d}(a) \perp \mathbf{d}(b)$.

\item[{\rm (RM3)}] In the case that $m$ and $n$ are both infinite we also require that only a finite number of entries in the matrix are non-zero.

\end{description}
We shall usually just say `rook matrix' instead of `generalized rook matrix'.

\begin{remark}\label{rem:quantales}
{\em Rook matrices, though not under this terminology, were first used by Hines \cite{Hines} in his work relating inverse semigroups to  linear logic.
The basic properties of rook matrices were sketched out by the second author during a visit to the University of Ottawa in 2013.
The motivation was to emulate the stabilization of $C^{\ast}$-algebras following some hints in  \cite{Grandis}.
This led to Proposition~\ref{prop:ale} and the inverse semigroups $M_{n}(S)$ and $M_{\omega}(S)$.
However, there was a whiff of adhocism about the definitions.
An attempt to dispell this was made by Wallis in his thesis \cite{Wallis} who developed a coordinate-free module-type theory
which led naturally to rook matrices in a way analogous to that in which linear transformations on a vector space lead to (classical) matrices.
Another approach to showing that the definition is a natural one uses some ideas from \cite{Resende} and \cite{KL}
where we refer the reader for any undefined terms.
We in fact generalize an argument to be found in \cite{Hines}.
Let $Q$ be an inverse quantal frame with top 1 and identity $e$.
We denote the involution by $a \mapsto a^{\ast}$.
Denote by $M_{n}(Q)$ the set of all $n \times n$ matrices over $Q$.
Then this is also an inverse quantal frame. 
The involution is transpose-and-$\ast$.
The multiplicative identity is the $n \times n$ identity matrix (where the identity is $e$).
The top element is the $n \times n$ matrix every element of which is 1.
The projections are the diagonal matrices whose diagonal entries are projections from $Q$.
We may characterize the partial units in $M_{n}(Q)$ as follows:
they are the elements $A$ such that $A^{\ast}A$ and $AA^{\ast}$ are both projections.
It quickly follows that these are precisely the matrices whose entries are partial units
and which satisfy the conditions (RM1) and (RM2).}
\end{remark}

When $n$ is finite, we use the notation $I_{n}$ to mean the $n \times n$ identity matrix.
Let $A$ be an $m \times n$ rook matrix and $B$ an $n \times p$ rook matrix.
The matrix $AB$ is defined as follows:
$$(AB)_{ij} = \bigvee_{k} a_{ik}b_{kj}.$$

\begin{proposition}\label{prop:ale} \mbox{} 
\begin{enumerate}

\item The matrix $AB$ is well-defined and is a rook matrix.

\item Multiplication is associative when defined.

\item The matrices $I_{n}$ are identities when multiplication is defined.

\item Let $A = (a_{ij})$ be a rook matrix.
Define $A^{\ast} = (a_{ji}^{-1})$.
Then $A^{\ast}$ is a rook matrix and $A = AA^{\ast}A$ and $A^{\ast} = A^{\ast}AA^{\ast}$.

\item The idempotents are those square rook matrices $E$ which are diagonal and whose diagonal entries are themselves idempotents.

\item If $A$ and $B$ are two rook matrices of the same size then $A \leq B$ if and only if $a_{ij} \leq b_{ij}$ for all $i$ and $j$.

\item If $A$ and $B$ are again of the same size then $A \perp B$ if and only if $a_{ij} \perp b_{ij}$;
if this is the case then $A \vee B$ exists and its elements are $a_{ij} \vee b_{ij}$.

\end{enumerate}
\end{proposition}
\begin{proof} We just prove (1). 
The other proofs are straightforward but the details may be found in \cite{Wallis}.
We prove first that for fixed $i$ and $j$, 
the join 
$\bigvee_{k} a_{ik}b_{kj}$
is defined.
Consider the two products $a_{ik}b_{kj}$ and $a_{il}b_{lj}$.
We calculate first $\mathbf{d}(a_{ik}b_{kj})\mathbf{d}(a_{il}b_{lj})$.
This product contains the term $b_{kj}b_{lj}^{-1}$ which is zero by (RM2) applied to $B$.
It follows that $\mathbf{d}(a_{ik}b_{kj})\mathbf{d}(a_{il}b_{lj}) = 0$.
Next we calculate $\mathbf{r}(a_{ik}b_{kj})\mathbf{r}(a_{il}b_{lj})$.
This product contains the term $a_{ik}^{-1}a_{il}$ which is zero by (RM1) applied to $A$.
It follows that $\mathbf{r}(a_{ik}b_{kj})\mathbf{r}(a_{il}b_{lj}) = 0$.
Hence the join is defined.
It remains to prove that $AB$ is a rook matrix.
We shall prove that (RM1) holds; the fact that (RM2) holds will then follow by symmetry. 
Fix $i$ and $j \neq k$.
Put $c_{ij} = \bigvee_{p} a_{ip}b_{pj}$ and $c_{ik} = \bigvee_{q} a_{iq}b_{qk}$.
We calculate $c_{ij}^{-1}c_{ik}$.
This is the join of terms of the form $b_{pj}^{-1}a_{ip}^{-1}a_{iq}b_{qk}$.
There are two kinds of terms.
If $p \neq q$ then $a_{ip}^{-1}a_{iq} = 0$ and the term disappears.
If $p = q$, then the term has the form $b_{pj}^{-1}a_{ip}^{-1}a_{ip}b_{pk} \leq b_{pj}^{-1}b_{pk} = 0$, and so also disappears.
\end{proof}

If $n$ is a finite non-zero natural number, define $M_{n}(S)$ to be the set of all $n \times n$ rook matrices over $S$.
This is a Boolean monoid.
In the case where $n = \omega$, we also write $M_{\omega}(S)$ for the set of all $\omega \times \omega$ rook matrices
over $S$, not forgetting that only a finite number of entries are non-zero.
This is a Boolean {\em semigroup}.

\begin{example}{\em 
Consider the case where $S$ is the inverse monoid $\{0,1\}$.
Then $M_{n}(\{0,1\})$ is the inverse semigroup of $n \times n$ rook matrices in the sense of Solomon \cite{Louis}, 
and so is isomorphic to the symmetric inverse monoid $I_{n}$.} 
\end{example}

\begin{example}{\em Define $I_{fin}(\mathbb{N})$ to be the inverse semigroup of partial bijections of $\mathbb{N}$ with finite domains.
Then $I_{fin}(\mathbb{N})$ is isomorphic to $M_{\omega}(\{0,1\})$.}
\end{example}

Let $a_{1}, \ldots, a_{n}$ be  a list of elements of $S$. 
Define $\Delta (a_{1}, \ldots,a_{n})$ to be the $n \times n$ diagonal matrix whose $n$ diagonal entries are precisely $a_{1}, \ldots, a_{n}$.
When we are working in $M_{\omega}(S)$, we shall use the notation  $\Delta_{\omega} (a_{1}, \ldots,a_{n})$ to be the
$\omega \times \omega$-diagonal matrix whose first $n$ diagonal entries are precisely $a_{1}, \ldots, a_{n}$ and all other entries are zero.
The following is trivial but useful.

\begin{lemma}\label{lem:anger} 
Let $S$ be a Boolean monoid.
Then $S$ is isomorphic to the local monoid $\Delta_{\omega} (1) M_{\omega}(S) \Delta_{\omega} (1)$  of $M_{\omega}(S)$.
\end{lemma}

In what follows, we shall regard $M_{\omega}(S)$ as being equipped with the distinguished idempotent $\Delta_{\omega} (1)$.

\begin{lemma}\label{lem:butterfly} Let $S$ be a Boolean monoid.
\begin{enumerate}

\item $S$ is equipped with an invariant mean if, and only if, the Boolean inverse semigroup $M_{\omega}(S)$ is equipped with
an invariant mean that normalizes $\Delta_{\omega} (1)$.

\item  $M_{\omega}(S)$ is orthogonally separating.

\end{enumerate}
\end{lemma}
\begin{proof} (1) By Lemma~\ref{lem:anger}, 
if $M_{\omega}(S)$ is equipped with an invariant mean that normalizes $\Delta_{\omega} (1)$, then $S$ is equipped with an invariant mean.
We prove the converse.
Suppose that $S$ is equipped with the invariant mean $\mu$.
Define  
$$\nu \colon E(M_{\omega}(S)) \rightarrow [0,\infty)$$ 
by 
$$\nu (\Delta_{\omega} (e_{1}, \ldots, e_{n})) = \sum_{i=1}^{n} \mu (e_{i}).$$
We only have to prove that if $\mathbf{e}$ and $\mathbf{f}$ are two idempotents in $M_{\omega}(S)$ 
such that $\mathbf{e} \, \mathscr{D} \, \mathbf{f}$ then $\nu (\mathbf{e}) = \nu (\mathbf{f})$.
In fact, it is enough to prove the following.
Let $A$ be an $n \times n$ rook matrix over $S$ such that
$$
A^{\ast}A = \Delta (e_{1}, \ldots, e_{n})
\mbox{ and }
AA^{\ast} = \Delta (f_{1}, \ldots, f_{n}).
$$
Then $\nu (\Delta (e_{1}, \ldots, e_{n})) = \nu  (\Delta (f_{1}, \ldots, f_{n}))$
where $\nu$ is defined in the obvious way.
We have that
$$\nu (A^{\ast}A)
=
\sum_{i=1}^{n} \sum_{j=1}^{n} \mu (\mathbf{d}(a_{ji}))$$
whereas
$$\nu (AA^{\ast})
=
\sum_{j=1}^{n} \sum_{i=1}^{n} \mu (\mathbf{r}(a_{ji})).$$
But these two numbers are equal since $\mu (\mathbf{d}(a_{ji})) = \mu (\mathbf{r}(a_{ji}))$.

(2) Let $E = \Delta_{\omega} (e_{1}, \ldots, e_{m})$ and $F = \Delta_{\omega} (0^{r},e_{1}, \ldots, e_{m})$
where $0^{r}$ simply means a sequence of $r$ 0's.
We claim that $E \, \mathscr{D} \, F$.
Let $A$ be the $\omega \times \omega$ matrix which consists of an $r \times m$ zero matrix
sitting on top of the diagonal matrix $E$ and then filled out with 0's.
It can be checked that $A$ has the property that
$A^{\ast}A = E$ and $AA^{\ast} = F$.
It follows that given two idempotents in $M_{\omega}(S)$ we may
`slide' the idempotents down the diagonal of one so that the resulting $\mathscr{D}$-related idempotent
is orthogonal to the other and then their join may be taken.
\end{proof}

\begin{center}
{\bf Proof of Theorem~\ref{thm:type-monoid}}
\end{center}

\noindent
{\bf Definition. }Let $S$ be a Boolean monoid.
Then the {\em type monoid}  $\mathsf{T}(S)$  of $S$ is defined to be the commutative monoid $\mathsf{E} (M_{\omega} (S))$.
Define $\delta \colon E(S) \rightarrow \mathsf{E} (M_{\omega} (S))$ by $\delta (e) = [\Delta_{\omega} (e)]$
and $\mathbf{u} =  \delta (1)$.\\

Proof of part (1).
On the basis of Proposition~\ref{prop:ale} and Lemma~\ref{lem:butterfly}, 
we have that $M_{\omega}(S)$ is an orthogonally separating Boolean inverse semigroup.
Thus by Proposition~\ref{prop:cakes}, $\mathsf{E} (M_{\omega} (S))$ is a commutative monoid
with identity the $\omega \times \omega$-zero matrix.
\begin{itemize}
\item $\delta (0)$ is the identity. This is immediate.
\item  $\delta (1)$ is an order-unit. 
By Proposition~\ref{prop:ale}, 
each idempotent in $M_{\omega}(S)$ is a diagonal matrix whose non-zero diagonal entries are idempotents.
Such a matrix is less than or equal to a matrix of the form $\Delta_{\omega}(1,\ldots,1)$ with, say, $n$ identities.
Such a matrix can be written as an orthogonal join of $n$ idempotents each of which has exactly one identity on the diagonal and zeros everywhere else.
The fact that all of these idempotents are $\mathscr{D}$-related follows from the proof of part (2) of Lemma~\ref{lem:butterfly}.
\item  If $e,f \in E(S)$ are othogonal then $\delta (e \vee f) = \delta (e) + \delta (f)$. The proof of this is straightforward.
\item If $e \, \mathscr{D} \, f$ then $\delta (e) = \delta (f)$. The proof of this is straightforward.
\end{itemize}
Proof of part (2).
This follows from part (1) of Lemma~\ref{lem:butterfly} and the fact that 
$$[\Delta_{\omega}(e_{1}, \ldots, e_{m})] = \delta (e_{1}) + \ldots + \delta (e_{m}).$$
This concludes the proof of Theorem~\ref{thm:type-monoid}.

\begin{remark}
{\em In  what follows, our bare construction of the map $\delta \colon E(S) \rightarrow \mathsf{E} (M_{\omega} (S))$  suffices, 
but it is an interesting question whether this map has universal properties linking partial structures to global structures
since by Proposition~\ref{prop:cakes}, 
the structure $(\mathsf{E}(S),\oplus)$ is a {\em partial} commutative monoid, 
leaving to one side the precise definition, 
whereas  $\mathsf{E} (M_{\omega} (S))$  is a commutative monoid {\em tout court}.
That there is such a universal characterization is hinted at in \cite[p.~142]{Wallis}
but we await the published account for an unambiguous presentation.}
\end{remark}

It now remains to  prove the main theorem of this paper, which requires the following definition.
Let $S$ be a Boolean monoid.
An $m \times (m+1)$ rook matrix $A$ over $S$ such that $A^{\ast}A = I_{m+1}$ is said to be a {\em Tarski matrix of degree $m$ over $S$}.
We shall also refer {\em tout court} to {\em Tarski matrices over $S$}.

\begin{example}\label{ex:question}{\em The existence of a Tarksi matrix of degree 1 is equivalent to the existence of 
a pair of elements $a$ and $b$ such that $\mathbf{d}(a) = 1 = \mathbf{d} (b)$ and $\mathbf{r}(a) \perp \mathbf{r}(b)$.}
\end{example}

The following lemma connects aspects of the structure of $S$ with aspects of the structure of its type monoid.

\begin{lemma}\label{lem:dinky} There exists a Tarski matrix of degree $n$ over $S$ 
if, and only if,
$(n+ 1) \mathbf{u} \leq n \mathbf{u}$ holds in the type monoid.
\end{lemma}
\begin{proof} Suppose that $A$ is a Tarski matrix of degree $n$.
Then $A^{\ast}A = I_{n+1}$ and $AA^{\ast} \leq I_{n}$.
Thus in $M_{\omega}(S)$, we have that 
$$\Delta_{\omega}(\overbrace{1, \ldots, 1}^{n+1}) \leq_{J} \Delta_{\omega}(\overbrace{1, \ldots, 1}^{n})$$
and so by Proposition~\ref{prop:cakes}, we have that
$$[\Delta_{\omega}(\overbrace{1, \ldots, 1}^{n+1})] \leq [\Delta_{\omega}(\overbrace{1, \ldots, 1}^{n})].$$
We now use the fact that
$[\Delta_{\omega}(\overbrace{1, \ldots, 1}^{m})] = m \mathbf{u}$.
The converse is proved essentially be reversing the above argument.
\end{proof}

The following theorem is termed the {\em Tarski alternatives} in the plural to distinguish it from the classical result that is proved later as Theorem~\ref{thm:TA}.

\begin{theorem}[The Tarski alternatives]\label{thm:TAS} Let $S$ be a Boolean monoid.
Then $S$ has an invariant mean if, and only if, there exists no Tarski matrix of any degree over $S$.
\end{theorem}
\begin{proof} By part (2) of Theorem~\ref{thm:type-monoid}, 
the monoid $S$ has an invariant mean if, and only if, the type monoid $T(S)$ admits a  morphism normalized at $\mathbf{u}$. 
By Theorem~\ref{them:wagon}, the type monoid admits a  morphism normalized at $\mathbf{u}$ if, and only if,  for no integer $n$ does the inequality $(n+1) \mathbf{u} \leq n \mathbf{u}$ hold. 
By Lemma~\ref{lem:dinky}, the inequality $(n+1) \mathbf{u} \leq n \mathbf{u}$ holds for some $n$ if, and only if, there exists a Tarski matrix of degree $n$.
\end{proof}

To conclude this section, we shall study a particular class of Boolean monoids that can be said to form the classical theory  of paradoxical decompositions.
An inverse subsemigroup of an inverse semigroup $S$ is said to be {\em wide} if it contains all the idempotents of $S$.
For the remainder of this section, we shall concentrate on those Boolean monoids that are wide inverse submonoids of symmetric inverse monoids $I(X)$.
Since the finite case is trivial from our point of view, we shall assume that $X$ is infinite.
Rook matrices over $I(X)$ have an alternative characterization that we describe below and which generalizes some results to be found in \cite{Hines}.
Let $X_{1}, \ldots, X_{m}$ be sets, not necessarily disjoint.
Define
$$\bigsqcup_{i=1}^{m} X_{i} = \bigcup_{i=1}^{m} X_{i} \times \{i\},$$
their disjoint union.

\begin{lemma}\label{lem:adams} Let  $S \subseteq I(X)$ be a Boolean monoid of partial bijections.
\begin{enumerate}

\item Let $X_{1}, \ldots, X_{n}, Y_{1}, \ldots, Y_{m} \subseteq X$
and let $f \colon \bigsqcup_{j=1}^{n} X_{j} \rightarrow  \bigsqcup_{i=1}^{m} Y_{i}$ be a bijection.
Define the bijection $f_{ij}$ from a subset of $X_{j}$ to a subset of $Y_{i}$ as follows. 
The domain of $f_{ij}$ consists of those elements $x \in X_{i}$ such that $f(x,j) \in Y_{i} \times \{i\}$.
We define $f_{ij}(x)$ to be that element $y \in Y_{j}$ such that $(y,i) = f(x,j)$.
Define $A$ to be the $m \times n$ matrix over $I(X)$ such that $A_{ij} = f_{ij}$.
Then $A$ is a rook matrix.

\item Let $A$ be an $m \times n$ rook matrix with entries $f_{ij}$.
For each $i$, 
define 
$X_{j} = \bigcup_{i=1}^{m} \mbox{\rm dom}(f_{ij})$ 
and for each $j$,
define
 $Y_{i} = \bigcup_{j=1}^{n} \mbox{\rm im}(f_{ij})$.
Define
$$f \colon \bigsqcup_{j=1}^{n} X_{j} \rightarrow  \bigsqcup_{i=1}^{m} Y_{i}$$ 
by $f(x,j) = (f_{ij}(x),i)$ if $(x,j) \in \mbox{\rm dom}(f_{ij})$.
Then $f$ is a bijection.

\item The constructions in parts (1) and (2) above are mutually inverse.

\end{enumerate}
\end{lemma}
\begin{proof} (1)  It is clear from the definition that the maps $f_{ij}$ are partial bijections.
Thus we may form the $m \times n$ matrix $A$ and it remains to show that it is a rook matrix.
Fix a column $j$ and consider the partial bijections $f_{1j}, f_{2j}, \ldots, f_{mj}$.
Then since $f$ is injective these domains must be disjoint.
Fix a row $i$ and consider the partial bijections $f_{i1}, \ldots, f_{in}$. 
Then since $f$ is injective these images must be disjoint.

(2) This is a straightforward verification.

(3) We refer back to the proof of part (1).
Fix a column $j$ and consider the partial bijections $f_{1j}, f_{2j}, \ldots, f_{mj}$.
Then the union of the domains of these functions must be $X_{j}$ since every element in $X_{j} \times \{j\}$ is mapped somewhere by $f$.
Fix a row $i$ and consider the partial bijections $f_{i1}, \ldots, f_{in}$. 
Then the union of the images of the partial bijections $f_{i1}, \ldots, f_{in}$ must be $Y_{i}$ since $f$ is surjective.
\end{proof}

The above result is important because it enables us to convert between questions about rook matrices and questions about bijections between sets.
We now suggest two distinct developments of these ideas.

\begin{remark}
{\em Let $A$ be an $m \times n$ rook matrix of partial bijections such that $A^{\ast}A = I_{n}$ and $AA^{\ast} = I_{m}$.
Such matrices are reminiscent of the dynamical systems studied in \cite{AE, AEK}.}
\end{remark}

\begin{remark}{\em  We resume the connection with quantales mentioned in Remark~\ref{rem:quantales}.
Let $Q$ be an inverse quantal frame \cite{Resende}.
Then it can be proved that the modules $Q^{m}$ and $Q^{n}$ are isomorphic if, and only if, there is an $m \times n$-rook matrix
with entries from $\mathsf{I}(Q)$, the inverse monoid of partial units of $Q$.
In addition, it can be proved that there is a Tarski matrix of degree $m$ over  $\mathsf{I}(Q)$
precisely when there is an injective morphism of $Q$-modules $Q^{m+1} \rightarrow Q^{m}$ 
mapping local sections to local sections, where we use terminology from \cite{RR}.
This all suggests a quantalic analogue of Leavitt path algebras \cite{BK, MT} 
that might provide an appropriate setting for these analogies, which however we shall not pursue here.}
\end{remark}

The following was proved in \cite{Kuratowski} in the case where $E = E'$.
The version given here is a trivial generalization where there is a bijection between $E$ and $E'$.

\begin{theorem}[Kuratowski's theorem]\label{thm:KP}
Let $E$ be a set partitioned thus $\{M,N\}$, where there is a bijection $\phi \colon M \rightarrow N$.
Let $E'$ be a set partitioned thus $\{P,Q\}$, where there is a bijection $\psi \colon P \rightarrow Q$.
Let $\alpha \colon E \rightarrow E'$ be a bijection.
Then there is a bijection from $M$ to $Q$.
In addition, this bijection is a finite join of compositions of the maps $\alpha$, $\phi$ and $\psi$ and partial identities defined on sets.
\end{theorem}

The following definition simply converts the above theorem into a property.
We say that a Boolean semigroup satisfies the {\em Kuratowski property} 
if given $e_{1} \vee e_{2}$,  an orthogonal join in which $e_{1} \, \mathscr{D} \, e_{2}$, 
and given 
$f_{1} \vee f_{2}$, an orthogonal join in which $f_{1} \, \mathscr{D} \, f_{2}$,
and given that
$(e_{1} \vee e_{2}) \, \mathscr{D} \, (f_{1} \vee f_{2})$
then $e_{1} \, \mathscr{D} \, f_{2}$.

\begin{lemma}\label{lem:dharma} Let $S$ be a Boolean monoid that is a wide inverse submonoid of an infinite symmetric inverse monoid.
\begin{enumerate}

\item In $M_{\omega}(S)$, we have that $\mathscr{D} = \mathscr{J}$.

\item In the type monoid of $S$, the algebraic preorder is an order.

\item The Boolean semigroup $M_{\omega}(S)$ satisfies the Kuratowski property.

\item In the type monoid of $S$, if $2a = 2b$ then $a = b$.

\end{enumerate}
\end{lemma}
\begin{proof} (1) Observe first that $\mathscr{D} = \mathscr{J}$ in $S$ because the Boolean algebra of idempotents of $I(X)$ has countable unions.
It is immediate that the semilattice of idempotents of $M_{n}(S)$ has countable joins and so $\mathscr{D} = \mathscr{J}$ holds in $M_{n}(S)$.
It is not true that the semilattice of idempotents of $M_{\omega}(S)$ has countable joins: 
for example, the set of idempotents $\Delta_{\omega} (1), \Delta_{\omega} (1,1), \Delta_{\omega} (1,1,1), \ldots$ has no join.
Nevertheless, in the inverse semigroup $M_{\omega}(S)$, we have that $\mathscr{D} = \mathscr{J}$ since 
$$M_{\omega}(S) = \bigcup_{i=1}^{\infty} \Delta_{\omega} (1^{i}) M_{\omega}(S)\Delta_{\omega} (1^{i})$$ 
and 
$\Delta_{\omega} (1^{i}) M_{\omega}(S)\Delta_{\omega} (1^{i}) \cong M_{i}(S)$.

(2) This is immediate by part (1) and Proposition~\ref{prop:cakes}.

(3) It is enough to prove that the Kuratowski property holds in $M_{n}(S)$.
Let $\mathbf{e}_{1},\mathbf{e}_{2}, \mathbf{f}_{1},\mathbf{f}_{2}$ be idempotents in $M_{n}(S)$
where 
$\mathbf{e}_{1} \perp \mathbf{e}_{2}$ 
and
$\mathbf{f}_{1} \perp \mathbf{f}_{2}$
and in addition
$\mathbf{e}_{1} \, \mathscr{D} \, \mathbf{e}_{2}$, witnessed by the rook matrix $A$,
$\mathbf{f}_{1} \, \mathscr{D} \, \mathbf{f}_{2}$, witnessed by the rook matrix $B$,
and finally
$\mathbf{e}_{1} \vee \mathbf{f}_{1} \, \mathscr{D} \, \mathbf{e}_{2} \vee \mathbf{f}_{2}$,
witnessed by the rook matrix $C$.
We now re-encode this data using the fact that the elements of our rook matrices are partial bijections and Lemma~\ref{lem:adams}.
Let
$$\mathbf{e}_{1} = \Delta (1_{E_{11}}, \ldots, 1_{E_{1n}}),$$
$$\mathbf{e}_{2} = \Delta (1_{E_{21}}, \ldots, 1_{E_{2n}}),$$
$$\mathbf{f}_{1} = \Delta (1_{F_{11}}, \ldots, 1_{F_{1n}}),$$
$$\mathbf{f}_{2} = \Delta (1_{F_{21}}, \ldots, 1_{F_{2n}}).$$
Using $\cong$ to mean that there is a bijection, we have that
$$E_{11} \sqcup \ldots \sqcup E_{1n} \cong  E_{21} \sqcup \ldots \sqcup E_{2n}$$ 
and
$$F_{11} \sqcup \ldots \sqcup F_{1n} \cong  F_{21} \sqcup \ldots \sqcup F_{2n}$$ 
and
$$W = (E_{11} \cup E_{21})  \sqcup \ldots \sqcup (E_{1n} \cup E_{2n})  \cong  (F_{11} \cup F_{21} ) \sqcup \ldots \sqcup (F_{1n}  \cup F_{2n}) = Z.$$ 
We now use the fact that orthogonal idempotents in $I(X)$ have disjoint domains of definition.
It follows that $E_{11} \cap E_{21} = \emptyset$ etc, and $F_{11} \cap F_{21} = \emptyset$ etc.
Thus $W$ is the disjoint union of
$$E_{11} \sqcup \ldots \sqcup E_{1n} \mbox { and }E_{21} \sqcup \ldots \sqcup E_{2n}$$
and that $Z$ is the disjoint union of
$$F_{11} \sqcup \ldots \sqcup F_{1n} \mbox{ and } F_{21} \sqcup \ldots \sqcup F_{2n}.$$
We therefore have exactly the set-up for an application of Kuratowski's theorem.
We deduce that
$$E_{11} \sqcup \ldots \sqcup E_{1n} \cong  F_{21} \sqcup \ldots \sqcup F_{2n}.$$ 
In addition, the bijection is obtained by taking the finite join of restrictions of the given bijections to sets.
It is here that we make use of the fact that $S$ is a wide inverse submonoid.
Thus $\mathbf{e}_{1} \, \mathscr{D} \, \mathbf{f}_{2}$, as required.

(4) This is immediate by part (3) and the definition of the type monoid.
\end{proof}

The following is the famous result of Tarski \cite[Theorem 16.12]{Tarski} that started our thinking on this subject.
For a more direct proof see \cite{HS}, but here we derive it from our main theorem.

\begin{theorem}[The Tarski alternative]\label{thm:TA} Let $S$ be a Boolean monoid that is a wide inverse submonoid of the symmetric inverse monoid $I(X)$.
Then either $S$ has an invariant mean or $S$ is strongly paradoxical.
\end{theorem}
\begin{proof} We begin with a couple of observations about the structure of $S$ and its type monoid.
First, since the Boolean algebra of idempotents of $S$ has all countable joins, 
it follows by  Lemma~\ref{lem:mary} and Lemma~\ref{lem:arden}, 
that $S$ is weakly paradoxical if, and only if, it is strongly paradoxical.
Second, by part (1) of Lemma~\ref{lem:dharma}, we have that $\mathscr{D} = \mathscr{J}$ in $M_{\omega}(S)$.
It therefore follows from part (7) of Proposition~\ref{prop:cakes}, that the algebraic preorder in the type monoid of $S$ is in fact an order.
We may now prove the theorem.
The easy direction is immediate: if $S$ has an invariant mean then by Lemma~\ref{lem:sam}, $S$ cannot be weakly paradoxical.
It remains to prove the hard direction.
Suppose that $S$ does not have an invariant mean.
Then by Theorem~\ref{thm:TAS}, it follows that $S$ has a Tarski matrix of degree $n$.
This is equivalent to the inequality holding $(n+1)\mathbf{u} \leq n\mathbf{u}$ in the type monoid of $S$ by Lemma~\ref{lem:dinky}.
Then, in fact, we have that $(n+1)\mathbf{u} = n\mathbf{u}$ since the algebraic order is a preorder.
It follows that $(n + k)\mathbf{u} = n\mathbf{u}$ for all $k = 1,2,3,\ldots$.
Thus for all $N \geq n$ we have that $N\mathbf{u} = n\mathbf{u}$.
Let $2^{l-1} \geq n$.
Then 
$2^{l}\mathbf{u} = 2^{l-1}\mathbf{u}$.
Because of our assumption on $S$,
we may use part (4) of Lemma~\ref{lem:dharma} to deduce that $2\mathbf{u} = \mathbf{u}$.
In particular, this means that there is Tarski matrix of degree 1 over $S$.
By Example~\ref{ex:question}, this is equivalent to $S$ being weakly paradoxical and so strongly paradoxical by our first observation above.
\end{proof}


\end{document}